\documentclass[12pt,twoside]{amsart}
\usepackage{amssymb}
\usepackage{amsfonts}
\usepackage{amsthm}

\newtheorem{theorem}{Theorem}
\newtheorem{lemma}[theorem]{Lemma}

\newcommand{\reals}{\mathbb{R}}
\newcommand{\complex}{\mathbb{C}}
\newcommand{\kp}[2]{\ker (\varphi-\lambda_{#1} I)^{#2}}

\begin{document}
\title[Finite dimensional invariant subspaces]{Characterization of finite dimensional  subspaces of complex
  functions that are invariant under linear differential operators}
\author{Pep Mulet}
\thanks{Dpt. matem\`atica aplicada, Univ. Val\`encia, Av. Vicent
  Andr\'es Estell\'es,  Burjassot, Spain, {\tt mulet@uv.es}.
This work has been financially supported by Spanish MINECO
  projects MTM2011-22741 and MTM2014-54388-P}
\date{\today}

\maketitle
\begin{abstract}
The method to  solve inhomogeneous linear differential equations that
is usually taught at school relies  on the fact that the right
hand side function is the product of a polynomial and an exponential
and that  the 
linear spaces of those functions are invariant  under differential
operators (finite or ordinary). 

This short note uses Jordan's canonical decomposition 
to prove that the linear spaces spanned by products of  polynomial and
exponentials are the 
only  linear complex spaces that are invariant under
differential operators, therefore  non-homogeneous linear finite difference
or ordinary differential  equations  can only be generically solved
when the right hand side belongs to those spaces.
\end{abstract}



\section{Introduction}
We characterize the finite dimensional subspaces of the space
of complex sequences which are invariant under every linear finite
differences operator as direct sums of spaces of arithmetic-geometric
sequences. We also characterize finite
dimensional subspaces of complex functions  which are
invariant under every linear differential operator as spaces of
polynomial-exponential spaces. This  explains why inhomogeneous
linear differential 
equations can only be generally solved  when the right hand sides are sums of
exponentials times polynomials.

The following result is equivalent to  Jordan's decomposition of a
matrix. Its proof is included for the sake of completeness.

\begin{theorem}\label{prop:1}
Let $p(x)=(x-\lambda_1)^{l_1}\dots(x-\lambda_s)^{l_s}$ be the minimal
polynomial of an endomorphism $\varphi$ of a finite dimensional
$\complex$-vector space $V$ (i.e.,
$p$ is the monic polynomial of minimal degree satisfying
$p(\varphi)=0$), $\lambda_i\in \complex$. Then 
\[
V=\oplus_{i=1}^{s}\kp {i}{l_i}.
\]
\end{theorem}

\begin{proof}
Define polynomials $p_i(x)=p(x)/(x-\lambda_i)^{l_i}$, whose greatest common divisor is
$1$. Therefore, by Bezout's identity (see \cite{hungerford}), there
exist polynomials $r_i$ 
such that $1=\sum_{i=1}^{s} r_ip_i$. This implies that the identity
mapping in $V$, $I_{V}$ can be written as
$I_{V}=\sum_{i=1}^{s} r_i(\varphi)p_i(\varphi)$, so
we deduce that  any $v\in V$ can be written as
\begin{equation}\label{eq:78}
v=\sum_{i=1}^{s} r_i(\varphi)p_i(\varphi)(v).
\end{equation}
Since
\begin{align*}
(\varphi-\lambda_i I)^{l_i}r_i(\varphi)p_i(\varphi)(v)
=
r_i(\varphi)p(\varphi)(v)=0,
\end{align*}
it turns out that $r_i(\varphi)p_i(\varphi)(v)\in \kp
{i}{l_i}$. Since $v$ is arbitrary, then
\eqref{eq:78} implies that
\begin{equation}\label{eq:2}
V=\sum_{i=1}^{s}\kp {i}{l_i}.
\end{equation}

Assume now that $\sum_{i} v_i=0$, $v_i\in\kp{i}{l_i}$. For any
$i=1,\dots,s$, Bezout's identity gives polynomials $s_i, t_i$, such
that
\[
s_i(x)(x-\lambda_i)^{l_i}+t_i(x)p_i(x)=1,
\]
which gives 
\[
s_i(\varphi)(\varphi-\lambda_i
I)^{l_i}+t_i(\varphi)p_i(\varphi)=I_{V}.
\]
Since $v_i=-\sum_{j\neq i}v_j$, then, for any $i=1,\dots,s$ we get
that $p_i(\varphi)v_j=0$, for $j\neq i$, which implies:
\begin{align*}
  v_i=s_i(\varphi)(\varphi-\lambda_i I)^{l_i}(v_i)-
  \sum_{j\neq   i}t_i(\varphi)p_i(\varphi)(v_j)=0,
\end{align*}
which yields that the sum \eqref{eq:2} is direct.
\end{proof}

\section{Invariant subspaces of complex sequences under linear finite  differences operators}
We consider the shift operator $S$ on complex sequences
$y=(y_n)=(y_n)_{n\in\mathbb N}$ given by
$(Sy)_n=y_{n+1}$, where $n$ will denote the independent variable
(index) unless otherwise stated.  Finite difference operators on
complex sequences are polynomials in $S$.
We denote by $\Pi_m$ the set of
(complex) 
polynomials of degree  at most $m$ and the subspace of complex
sequences 
\[
G_{\lambda,m}=\{(\lambda^n p(n)) \colon p\in\Pi_{m-1}\}.
\]

\begin{lemma}\label{lemma:5}
  Given $\lambda\in \complex$ and $m>0$, 
  $\ker (S-\lambda I)^m=G_{\lambda,m}$.
\end{lemma}
\begin{proof}
  By induction on $k$ it can be easily established that there exist
  $\alpha_{j}^{k, r}\in\mathbb C$ such that
  \begin{equation}\label{eq:23}
    (S-\lambda I)^k(n^r\lambda^n)=\sum_{j=0}^{r-k}\alpha_{j}^{k, r}
    (n^j\lambda^n),\quad \alpha_{r-k}^{k,r}\neq 0,
  \end{equation}
  for any $r\geq k$ and 
  $(S-\lambda I)^k(n^r\lambda^n)=0$ for $k>r$. This immediately gives
  that   $G_{\lambda,m} \subseteq \ker (S-\lambda I)^m$, for all $m$.
  We prove the other inclusion by induction on $m$, the case $m=1$ being
  trivial. So, assume that $\ker (S-\lambda I)^m=G_{\lambda,m}$ and
  aim to prove $\ker (S-\lambda I)^{m+1}=G_{\lambda,m+1}$. For this,
  consider $y\in\ker (S-\lambda I)^{m+1}$, so that $(S-\lambda
  I)^{m}y\in\ker (S-\lambda
  I)=G_{\lambda,1}=\{(\alpha\lambda^n) \colon \alpha\in\complex\}$ and,
  therefore, 
  \begin{equation}\label{eq:231}
    (S-\lambda   I)^{m}y=\alpha (\lambda^n),
  \end{equation}
  for some $\alpha\in\complex$. On the other hand,
  by \eqref{eq:23} 
  \begin{equation}\label{eq:232}
    (S-\lambda I)^m(n^m\lambda^n)=\alpha_{0}^{m,m}(\lambda^n),
  \end{equation}
  which, together with \eqref{eq:231}, gives:
  \[
  (S-\lambda I)^m(y-\frac{\alpha}{\alpha_{0}^{m,m}}(n^m\lambda^n))=0.
  \]
   The induction hypothesis thus yields
  $y-\frac{\alpha}{\alpha_{0}^{m,m}}(n^m\lambda^n)\in G_{\lambda,m}$, that
  is $y\in G_{\lambda,m+1}$ and the  proof is complete.

\end{proof}

\begin{theorem}\label{theorem:3}
  Let $V$ be a finite dimensional subspace of complex sequences. Then
  $V$ is invariant under every linear finite difference operator if and
  only if there exists $\lambda_i, l_i$, $i=1,\dots,s$ such that
  $V=\oplus_{i=1}^{s}  G_{\lambda_i,l_i}$.
\end{theorem}  

\begin{proof}
  By Lemma \ref{lemma:5}, $G_{\lambda,m}$ is invariant under $S$ for any
  $\lambda, m$, thus any subspace of the form $\oplus_{i=1}^{s}
  G_{\lambda_i,l_i}$ is also $S$-invariant. Since the linear finite
  difference operators are polynomials in $S$, then those subspaces
  are invariant under those difference operators.

  On the other hand, if $V$ is a finite dimensional subspace which is
  invariant under every linear finite difference operator, in particular
  it is invariant under $S$. Therefore, by Proposition \ref{prop:1}
  there exist $\lambda_i, l_i$, $i=1,\dots,s$ such that
  \[
  V=\oplus_{i=1}^{s} \ker (S|_V-\lambda_i I_V)^{l_i}=\oplus_{i=1}^{s}
  (\ker (S-\lambda_i I)^{l_i}\cap V),\quad 
  \ker (S-\lambda_i I)^{l_i}\cap V\neq 0,
  \]
  and we can assume that  $l_i$ is the smallest integer satisfying
  this equation. 
  Since Lemma \ref{lemma:5} implies
   \[
   V=\oplus_{i=1}^{s} (V\cap G_{\lambda_i,l_i}).
   \]
   the proof will    be complete if we establish
   $G_{\lambda_i,l_i}\subseteq V$. Since 
   $$V\cap G_{\lambda_i,l_i-1}\subsetneq V\cap
   G_{\lambda_i,l_i}\neq 0,
   $$
   we  can choose
   \begin{equation}\label{eq:233}
     v=\sum_{s=0}^{l_i-1}   \beta_{s} (n^s\lambda^n)\in V\cap
   G_{\lambda_i,l_i}\setminus V\cap
   G_{\lambda_i,l_i-1},
   \end{equation}
   i.e.     $\beta_{l_i-1}\neq 0$.
   Equation \eqref{eq:23} and the $S$-invariance of $V$ yield for any
   $k\geq 0$:
   \begin{equation*}
    (S-\lambda I)^k\sum_{s=0}^{l_i-1}
   \beta_{s}
   (n^s\lambda^n)=\sum_{s=0}^{l_i-1}\beta_s\sum_{j=0}^{s-k}\alpha_{j}^{k,s} 
    (n^j\lambda^n)=
    \sum_{j=0}^{l_i-1-k}
    \gamma_{j,k}
    (n^j\lambda^n)\in V,
  \end{equation*}
  with $\gamma_{j,k}=\sum_{s=j+k}^{l_i-1}\beta_s\alpha_{j}^{k,s}
  $. Since, for any $0\leq k\leq
  l_{i}-1$, we get 
  $\gamma_{l_i-1-k,k}=\beta_{l_i-1}\alpha_{l_{i}-1-k}^{k,l_{i}-1}\neq
  0$ by \eqref{eq:23} and \eqref{eq:233}, we deduce that $G_{\lambda_i,l_i} \subseteq V$, as claimed.   
\end{proof}  

\section{Invariant subspaces of complex functions
  under linear  differential operators}
We consider now the differential operator $D$ on functions $y\colon
\mathbb K\to\complex$ ($\mathbb K=\reals$ or $\complex$)  given by
$Dy=y'$. Linear differential operators are polynomial evaluations of $D$.
We denote  by 
\[
H_{\lambda,m}=\{ e^{\lambda t}p(t) \colon p\in\Pi_{m-1}\}.
\]
\begin{lemma}\label{lemma:6}
  Given $\lambda\in \complex$ and $m>0$, 
  $\ker (D-\lambda I)^m=H_{\lambda,m}$.
\end{lemma}
\begin{proof}
  The proof is similar to that of Lemma \ref{lemma:5} and relies on
  the fact, easily established   by induction on $k$, that  there exist
  $\alpha_{j}^{k, r}\in\mathbb C$ such that
  \begin{equation}\label{eq:33}
    (D-\lambda I)^k(t^re^{\lambda t})=\sum_{j=0}^{r-k}\alpha_{j}^{k, r}
    (t^je^{\lambda t}),\quad \alpha_{r-k}^{k,r}\neq 0,
  \end{equation}
  for any $r\geq k$ and 
  $(D-\lambda I)^k(t^re^{\lambda t})=0$ for $k>r$. 
\end{proof}

The proof of the following theorem relies on Lemma \ref{lemma:6} and
is similar to that of Theorem 
\ref{theorem:3}. It  explains why inhomogeneous linear differential
equations can only be generally solved  when the right hand sides are sums of
exponentials times polynomials (see \cite{Strang,Gear,Hairer}).

\begin{theorem}
  Let $V$ be a finite dimensional subspace of the space of functions
  $y\colon\mathbb K\to \complex$. Then
  $V$ is invariant under every linear  differential operator if and
  only if there exists $\lambda_i, l_i$, $i=1,\dots,s$ such that
  $V=\oplus_{i=1}^{s}  H_{\lambda_i,l_i}$.
\end{theorem}

\end{document}